\documentclass[reqno,11pt]{amsart}
\usepackage{lineno,amsmath,amsthm,amssymb,graphics,epsfig,mathrsfs}
\usepackage{enumerate,geometry,etoolbox, scalerel, stackengine, relsize}
\usepackage{amscd,latexsym,url,upgreek,mathtools,color}
\usepackage{mathrsfs}
\usepackage[english]{babel}

\stackMath
\newcommand\reallywidehat[1]{%
\savestack{\tmpbox}{\stretchto{%
  \scaleto{%
    \scalerel*[\widthof{\ensuremath{#1}}]{\kern-.6pt\bigwedge\kern-.6pt}%
    {\rule[-\textheight/2]{1ex}{\textheight}}
  }{\textheight}%
}{0.5ex}}%
\stackon[1pt]{#1}{\tmpbox}%
}

\newtheorem{theorem}{Theorem}[section]
\newtheorem{lemma}[theorem]{Lemma}

\newtheorem{proposition}[theorem]{Proposition}
\newtheorem{example}[theorem]{Example}

\theoremstyle{definition}
\newtheorem{remark}[theorem]{{\bf Remark}}
\newtheorem{definition}[theorem]{Definition}
\newtheorem{problem}[theorem]{Conjecture}

\def\C{\mathbb C}

\def\R{\mathbb R}

\def\N{\mathbb N}

\def\Z{\mathbb Z}

\def\({{\rm (}}
\def\){{\rm )}}

\begin{document}

\numberwithin{equation}{section}

\title{Analyticity and supershift with irregular sampling}

\author[F. Colombo]{F. Colombo}
\address{(FC) Politecnico di
Milano\\Dipartimento di Matematica\\Via E. Bonardi, 9\\20133 Milano\\Italy}
\email{fabrizio.colombo@polimi.it}

\author[I. Sabadini]{I. Sabadini}
\address{(IS) Politecnico di
Milano\\Dipartimento di Matematica\\Via E. Bonardi, 9\\20133 Milano\\Italy}
\email{irene.sabadini@polimi.it}

\author[D. C. Struppa]{D. C. Struppa}
\address{(DCS) The Donald Bren Presidential Chair in Mathematics\\ Chapman University, Orange, CA 92866 \\ USA}
\email{struppa@chapman.edu}

\author[A. Yger]{A. Yger}
\address{(AY) IMB, Universit\'e de Bordeaux, 33405, Talence, France}
\email{yger@math.u-bordeaux.fr}


\begin{abstract}
The notion of supershift generalizes that one of superoscillation and expresses the fact that the sampling of a function in an interval allows to compute the values of the function outside the interval. In a previous paper we discussed the case in which the sampling of the function is regular and we are considering supershift in a bounded set, while here we investigate how irregularity in the sampling may affect the answer to the question of whether there is any relation between supershift and real analyticity on the whole real line. We show that the restriction to $\R$ of any entire function displays supershift, whereas the converse is, in general, not true. We conjecture that the converse is true as long as the sampling is regular, we discuss examples in support and we prove that the conjecture is indeed true for periodic functions.

\end{abstract}
\maketitle



\section{Introduction and preliminary definitions}

The notion of supershift (in itself a generalization of the notion of superoscillation arising in quantum mechanics) expresses the fact that the sampling of a function in an interval allows to compute the values of the function far from the interval.
In a recent paper, see \cite{cssy23}, we studied the relation between supershift and real analyticity.
 In particular, we used a classical result due to Serge Bernstein to
show that real analyticity for  a complex valued function implies a strong form of supershift. On the other hand, we used a
parametric version of a result by
Leonid Kantorovitch to show that the converse is not true in the sense that we can construct an example of a smooth function that exhibit supershift on a bounded set and yet is not real analytic on that same set.

In \cite{cssy23} we restricted our attention to what we called a regular sampling. In this paper, instead, we study the more general case where the sampling points can be distributed in an interval with no regularity, but we consider the more difficult situation in which supershift occurs on the whole real line.

The first class of results, see
 Section \ref{sect3}, shows the superoscillatory phenomenon arises  even in the case of irregular samplings in a way generalizing the standard one.  The most general configuration of sampling points is treated via Lagrange-Hermite polynomials, see Theorem \ref{sect3-prop2} and Theorem \ref{Thm27}, even though a well-known example of Erd\"os and V\'ertesi shows that such polynomials fail to approximate continuous functions almost everywhere (and in particular it is always possible to find continuous functions for which the supershift fails even for regular samplings).
This leads to the fact that the restriction to $\R$ of any entire function displays supershift phenomena, see Theorem \ref{sect4-thm2} which gives, using Proposition \ref{sect2-prop1bis} and Theorem \ref{sect3-prop2} respectively, Example \ref{sect4-expl1} and Theorem \ref{Thm27}.

  This leads us naturally to
Section \ref{sect4} where we obtain some results that suggest the following conjecture (Conjecture \ref{sect4-conj1}): {\it is any regular $\C$-supershift on $\R$ (in the sense of Definition {\rm \ref{sect2new-def1}}) necessarily the restriction of an entire function}?
Proposition \ref{sect4-prop2} and Remark  \ref{sect4-rem2}, which are inherited from Kantorovich's example, suggest that the regularity hypothesis on the sampling is natural for our conjecture to hold.\\
More precisely, Proposition  \ref{sect4-prop2} shows that slightly irregular samplings on $[-1,1]$ allow the construction of non real analytic functions on $\R$ which can be extrapolated, uniformly on any compact set, from their sample values on $[-1,1]$.

As a support for Conjecture \ref{sect4-conj1}, we prove it for periodic regular $\C$-supershifts on
$\R$, see Theorem \ref{sect4-thm1}.

We now review some preliminary definitions which will be useful in the sequel.

Let
\begin{equation}\label{sect1-eq0}
H = \begin{bmatrix}
& h_{1,0},...,h_{1, \nu(1)}\\
& h_{2,0},h_{2,1},..., h_{2, \nu(2)}\\
& \vdots \\
& h_{N,0},h_{N,1},h_{N,2},...,h_{N,\nu(N)}\\
& \vdots  \end{bmatrix}
\end{equation}
be a collection of real numbers (interpreted as frequencies) which all belong to the closed interval $[-1,1]$ {\rm with $h_{N,\nu} \geq h_{N,\nu+1}$ for $N\in \N^*$ and $0\leq \nu \leq \nu(N)-1$}.

\begin{definition}
A sequence of generalized trigonometric polynomials $(T_{H,N})_{N \geq 1}$  of the form
$$
{T_{H,N}}(x) = \sum\limits_{\nu=0}^{\nu(N)} C_{H,\nu}(N)\, \exp( ix h_{N,\nu}),\quad ({\rm with}\ C_{H,\nu}(N)
\in \C \quad {\rm for}\ N \in \N^*, \ 0\leq \nu\leq \nu(N))
$$
is called an {\it $H$-sequence of generalized trigonometric polynomials}.
\\
Given an open subset $A\subset \R$, an {\it $(A,H)$-sequence of generalized trigonometric polynomials}
$\{T_{H,N}[a]:a\in A\}_{N\geq 1}$ is by definition an $A$-parametrized $H$-sequence of generalized trigonometric polynomials such that each complex amplitude $C_{H,\nu}(N,a)$
depends continuously on
$a\in A$, for any $N\in \N^*$ .
\end{definition}
\begin{definition}[Superoscillating sequence] An $(A,H)$-sequence of generalized trigonometric polynomials is said to be {\it superoscillating} if there are two continuous functions $g_H~: A \rightarrow \C$ and $C_H : A \rightarrow \R$ such that both $V=\{a \in A\,:\, C_H(a) \not=0,\ |g_H(a)| > 1\}$ and the open subset $U$ of points $x\in \R$ about which
\begin{equation}\label{sect1-eq1ter}
\lim_{N\rightarrow +\infty} T_{H,N}[a]\, (x) = C_H(a)\, \exp (i\, g_H(a)\, x),
\end{equation}
locally uniformly with respect to $(a,x)$, are non-empty open subsets respectively of $A$ and
$\R$.  The set $U$ is then called {\it the superoscillating subset} of the $(A,H)$-superoscillating sequence $\{T_N[a]\,:\, a\in A\}_{N\geq 1}$.
\end{definition}

\begin{definition}[Regular sampling]
We will call regular sampling of the frequency interval $[-1,1]$ the sampling defined by $H^{\boldsymbol \epsilon}=[h_{N,\nu}^{\boldsymbol \epsilon}]$ where
$\boldsymbol \epsilon= (\epsilon_N)_{N\geq 1}$ is a sequence of elements in $[0,1[$ which tends to $0$ when $N$ tends to infinity and
\begin{equation}\label{sect1-eq3}
h_{N,\nu}^{\boldsymbol \epsilon} = 1 - 2\, \Big(\frac{\nu + \epsilon_N (N-\nu)}{N}\Big), \ 0\leq \nu \leq N~.
\end{equation}
\end{definition}
\begin{example}
The $(\R,H^{\boldsymbol \epsilon})$-sequence $\{T_N^{\boldsymbol \epsilon} [\lambda]\,:\, \lambda\in \R\}_{N\geq 1}$, where
\begin{equation}\label{sect1-eq4}
\begin{split}
T_{N}^{\boldsymbol \epsilon}[\lambda]\, (x) & = \sum\limits_{\nu=0}^N \binom{N}{\nu}\, \Big(\frac{1+\lambda}{2}\Big)^{N-\nu}
\Big(\frac{1-\lambda}{2}\Big)^{\nu}
\exp \big( i h_{N,\nu}^{\boldsymbol \epsilon}\, x\big) \\
& = \exp (i\, \epsilon_N x)\, \Big( \cos
\Big(x\, \frac{1-\epsilon_N}{N}\Big) + i\, \lambda\, \sin \Big(x\, \frac{1-\epsilon_N}{N}\Big)\Big)^N
\end{split}
\end{equation}
is superoscillating on $U=\R$.
\end{example}
\noindent
From now on, we will mostly consider $(\mathbb{R},H)$-superoscillating sequences.\\
The definition of superoscillating functions can be extended to include more general cases as follows:
\begin{definition}\label{sect1-def1}(Supershift Property $({\rm SP})_{\mathscr F}$)
Let $A$ be an open interval of $\R$, possibly $\R$ itself, which contains $[-1,1]$ and $\psi : a \in A \longmapsto \psi_a\in \mathscr F$ be a continuous map from
$A$ to a topological $\C$-vector space $\mathscr F$.
 Let
$$\left\{{T_{H,N}}[\lambda](x) = \sum\limits_{\nu=0}^{\nu(N)} C_{H,\nu}(N,\lambda)\, \exp( ix h_{N,\nu})\,:\, \lambda\in \R\right\}_{N\geq 1}
$$
be an
$(\R,H)$-superoscillating sequence and assume that all the complex functions $C_{H,\nu}(N,\lambda)$ are continuous in $\lambda$. The map
$\psi$ is said to satisfy the {\it Supershift Property} $({\rm SP})_{\mathscr F}$ on $A$ with respect to $\{T_{H,N}[\lambda]\,:\, \lambda\in \R\}_{N\geq 1}$ if the sequence of functions
\begin{equation}\label{sect1-eq6}
a \in A \longmapsto \sum\limits_{\nu=0}^{\nu(N)}
C_{H,\nu} (N,a)\, \psi_{h_{N,\nu}} \in \mathscr F,\quad N=1,2,...
\end{equation}
converges to $\psi$ in the space of continuous functions $\mathcal C(A,\mathscr F)$, with respect to the topology of uniform convergence on any compact subset, i.e.
$$
\lim_{N\to\infty}\sum\limits_{\nu=0}^{\nu(N)}
C_{H,\nu} (N,a)\, \psi_{h_{N,\nu}}=\psi_a .
$$
\end{definition}
\begin{remark}
The reader will notice that when
$$C_\nu(N,a):= {\binom{N}{\nu}\Big(\frac{1+a}{2}\Big)^{N-\nu} \Big(\frac{1-a}{2}\Big)^{\nu}}
$$
 this definition reduces to
 the original example of Aharonov, where  $A={\mathbb{R}}$, {$H=H^{\boldsymbol 0}:= \{1-2\nu/N\,:\, N\in \N^*,\ 0\leq \nu \leq N\}$}.
\end{remark}
The  two definitions below were originally given in \cite{cssy23}:
\begin{definition}\label{sect1-def2} (Translation-Commuting Supershift Property $({\rm TCSP})_{\mathscr F}$)
Let $A$ be as above and for $a \in A$, let $a\in A\mapsto \psi_a\in \mathscr F$ be a continuous map from
$A$ to $\mathscr F$. Let $\{T_{H,N}[\lambda]\,:\, \lambda\in \R\}_{N\geq 1}$ be a
$(\R,H)$-superoscillating sequence. The continuous map
$\psi$ is said to satisfy the {\it $\mathscr F$-Translation-Commuting Supershift Property $({\rm TCSP})_{\mathscr F}$} on $A$ with respect to the superoscillating sequence $\{T_{H,N}[\lambda]\,:\, \lambda\in \R\}_{N\geq 1}$ if the sequence of functions
\begin{equation}\label{sect1-eq8}
\left(\sum\limits_{\nu=0}^{\nu(N)}
C_{H,\nu} (N,a)\, \psi_{a' + h_{N,\nu}}\right) \subset \mathscr F,\quad N=1,2,...
\end{equation}
defined for
$$
(a,a') \in \mathbb A := \big\{(a,a')
\in \R \times A\,:\, a'+ [-1,1] \subset A,\ a+a'\in A\big\}
$$
converges to $\psi_{a + a'}$ in $\mathcal C(\mathbb A,\mathscr F)$ with respect to the topology of uniform convergence on any compact subset of $\mathbb A$.
\end{definition}

\begin{definition}\label{sect2new-def1} (Regular $\C$-supershift)
Let $A$ be an open interval of $\R$ with length $R$ strictly larger than $2$. A continuous map
$\psi: A \rightarrow \C$ is called a {\it regular $\C$-supershift} if the two following conditions are fulfilled.
\begin{enumerate}
\item The map $\psi$ satisfies the $({\rm TCSP})_\C$ property with respect to any superoscillating sequence $\{T_{N}^{\boldsymbol \epsilon}[\lambda]\,:\, \lambda\in \R\}_{N\geq 1}$, as in
\eqref{sect1-eq3} and \eqref{sect1-eq4}, where $\boldsymbol \epsilon$ is any sequence in
$[0,1[^{\N^*}$ which tends to $0$ when $N$ tends to infinity.
\item Given $\{(\epsilon_{\iota',N})_{N\geq 1}\,:\, \iota \in I'\}$ a family of such sequences, such that the convergence of all sequences $\boldsymbol \epsilon_{\iota'}$ towards $0$ is uniform with respect to the index $\iota'$, the convergence  for $(a,a') \in \mathbb A$ of the polynomial functions
\begin{multline}\label{sect2new-eq10}
 \sum\limits_{\nu=0}^{N}
\binom{N}{\nu}
\Big(\frac{1+a}{2}\Big)^{N-\nu}
\Big(\frac{1-a}{2}\Big)^\nu
\psi\Big( a' +
\Big(1 - 2\, \Big(\frac{\nu + \epsilon_{\iota',N}\, (N-\nu)}{N}\Big)\Big)\Big),\quad N=1,2,...
\end{multline}
to $\psi(a+a')$ in $\mathcal C(\mathbb A,\C)$ is uniform with respect to the index $\iota'$ over the compact subsets of $\mathbb A$.
\end{enumerate}
\end{definition}

\section{Superoscillating sequences attached to irregular sampling in $[-1,1]$}\label{sect3}

In this section we consider superoscillating sequences with irregular sampling and we will prove their convergence in the space ${\rm Exp}(\C)$.

\begin{definition}
We consider a collection of elements in $[0,1]$:
$$
\boldsymbol \Xi = \{\xi_{N,k}\,:\, N\in \N^*,\ 1\leq k\leq N\}
$$
Let $A_{N,k}$ be the set consisting of the two opposite points $\pm \frac{1}{N}
\Big( 1 - \frac{\xi_{N,k}}{N}\Big)$
and consider the Minkowski set addition
$$
A_{N,1} + \ldots + A_{N,N}
$$
that consists of at most $2^N$ distinct points, all of them contained in the interval $[-1,1]$.
After reordering them in decreasing order, and using the Minkowski set addition, we define
$$
\{h_{N,\nu}^{\boldsymbol \Xi}\,:\, \nu=0,...,\nu(N)\leq 2^N\} = \sum\limits_{k=1}^N A_{N,k}.
$$
The sampling defined by this set
is said {\it almost regular}.
\end{definition}

\begin{definition}
We say that the sampling of the domain $[-1,1]$ defined by the matrix $H$ in \eqref{sect1-eq0} is  {\it irregular} if $
\forall\, N \in \N^*,\ \nu(N)=N,
$ and infinitely many rows $H_{N_\iota}$ of
$H$ are such that $\nu \in \{0,..., N_\iota -1\} \mapsto h_{N_\iota,\nu}- h_{N_\iota,\nu+1}>0$ is not constant.
\end{definition}
To generalize
the $(\R,H^{\boldsymbol \epsilon})$-sequence $\{T_N^{\boldsymbol \epsilon} [\lambda]\,:\, \lambda\in \R\}_{N\geq 1}$ we consider a collection $\boldsymbol \Xi$ of elements in $[0,1]$ corresponding to the almost regular sampling above,
and
$
\{h_{N,\nu}^{\boldsymbol \Xi}\,:\, \nu=0,...,\nu(N)\}.
$
The $(\R,H^{\boldsymbol \Xi})$-sequence $\{T_N^{\boldsymbol \Xi}[\lambda]\,:\, \lambda\in \R\}_{N\geq 1}$, where
\begin{equation}\label{sect1-eq5}
T_{N}^{\boldsymbol \Xi}[\lambda]\, (x) =
\prod_{k=1}^N \Big(
\frac{1+\lambda}{2}\, \exp \Big( \frac{ix}{N}\Big(1- \frac{\xi_{N,k}}{N}\Big)\Big) +
\frac{1-\lambda}{2}\, \exp \Big(- \frac{ix}{N}\Big( 1 - \frac{\xi_{N,k}}{N}\Big)\Big)\Big)
\end{equation}
is superoscillating  (with $g_{H^{\boldsymbol \Xi}} (\lambda) = \lambda$ and $C_{H^{\boldsymbol \Xi}}(\lambda)\equiv 1$) on $U=\R$. \\
When the matrix $H$ in \eqref{sect1-eq0} is such that
$
\forall\, N \in \N^*,\ \nu(N)=N,
$
the sequence of Lagrange-Hermite interpolators (with respect to the parameter $\lambda$)
\begin{equation}\label{sect1-eq5bis}
\Big\{ x \longmapsto T_{H,N}^{\rm Lag}[\lambda](x) := \sum\limits_{\nu=0}^N
\Big(\prod_{\nu'\not=\nu}
\frac{\lambda - h_{N,\nu'}}{h_{N,\nu} - h_{N,\nu'}}\Big)\, e^{ih_{N,\nu} x}\,:\, \lambda\in \R\Big\}_{N\geq 1}
\end{equation}
is also superoscillating with respect to $g: \lambda\in \R \rightarrow \lambda$ and $C: \lambda\in \R  \rightarrow 1$ since
\begin{equation}\label{sect1-eq5ter}
\begin{split}
\Big| e^{i\lambda x} -  T_{H,N}^{\rm Lag}[\lambda](x)\Big| \leq \frac{\prod_{\nu=0}^N |\lambda-h_{N,\nu}|}{(N+1)!}
\sup_{|\xi|\leq \max (1,|\lambda|)} \Big|
\Big(\frac{\partial}{\partial \lambda}\Big)^{N+1} [e^{i\lambda x}] (\xi)\Big| \leq
\frac{\big((|\lambda|+1)|x|\big)^{N+1}}{(N+1)!}
\end{split}
\end{equation}
for any $(\lambda,x)\in \R^2$ according to the expression of the reminder term in Lagrange interpolation formula, see \cite[Theorem 2.2]{ACSSST21}. Such superoscillating sequence $\{T_{H,N}^{\rm Lag}[\lambda]\,:\, \lambda\in \R\}_{N\geq 1}$
may correspond to {\it irregular sampling} of the low-frequency domain $[-1,1]$.\\
Using the sequence \eqref{sect1-eq5bis} in which we use $z\in\mathbb C$ instead of the real variable $x$, we have:
\begin{proposition}\label{sect2-prop1bis}
For any almost regular sampling
$
\boldsymbol \Xi
$
as above and any $a\in \R$, the sequence of entire functions with exponential growth
\begin{equation}\label{sect3-prop1bis-eq1}
\Big( z \longmapsto T_{N}^{\boldsymbol \Xi}[a]\, (z) \Big)_{N\geq 1}
\end{equation}
converges in ${\rm Exp}(\C)$ towards $z \longmapsto e^{iaz}$. Moreover, there exists a compact $K_\R$ in $\R$ such that the convergence is uniform
with respect both to $a\in K_\R$ and to the collection $\boldsymbol \Xi$. As a consequence the $(\R,H^{\boldsymbol \Xi})$-sequence $\{T_N^{\boldsymbol \Xi}[a]\,:\, a\in \R\}_{N\geq 1}$ is superoscillating.
\end{proposition}

\begin{proof} We follow here the method that we previously used in \cite[Lemma 4.1]{CSSY22}.
For each $a\in \R$, $z\in \C$, $\xi \in [0,1]$, one has
\begin{equation*}
\begin{split}
\Big|\frac{1+a}{2}\, \exp \Big( \frac{iz}{N}\Big(1- \frac{\xi}{N}\Big)\Big) & +
\frac{1-a}{2}\, \exp \Big(- \frac{iz}{N}\Big( 1 - \frac{\xi}{N}\Big)\Big)\Big|
\\
& = \Big|\cos \Big(
\frac{z}{N} \Big(1-\frac{\xi}{N}\Big)\Big) + ia\, \sin \Big(
\frac{z}{N} \Big(1-\frac{\xi}{N}\Big)\Big)\Big| \\
& =
\Big|\cos \Big(
\frac{z}{N} \Big(1-\frac{\xi}{N}\Big)\Big) + ia\frac{z}{N} \Big(1- \frac{\xi}{N}\Big)
\, {\rm sinc}\, \Big(\frac{z}{N} \Big(1-\frac{\xi}{N}\Big)\Big)\Big| \\
& \leq \Big( 1 + \frac{|a| \, |z|}{N}\Big)\, \exp \Big(\frac{|{\rm Im}\, z|}{N}
\Big(1-\frac{\xi}{N}\Big)\Big) \\
&\leq \Big( 1 + \frac{|a| \, |z|}{N}\Big)\,
\exp \Big(\frac{|{\rm Im}\, z|}{N}\Big)
\end{split}
\end{equation*}
since the entire function sinus cardinal $z\mapsto {\rm sinc}\, (z) = \sin z/z$ satisfies
in the whole complex plane the upper estimate $|{\rm sinc}\, z| \leq \exp (|{\rm Im}\, z|)$ in $\C$.
As a consequence, one has
\begin{multline*}
|T_{N}^{\boldsymbol \Xi}[a]\, (z)| =
\Big|\prod_{k=1}^N \Big(
\frac{1+a}{2}\, \exp \Big( \frac{iz}{N}\Big(1- \frac{\xi_{N,k}}{N}\Big)\Big) +
\frac{1-a}{2}\, \exp \Big(- \frac{iz}{N}\Big( 1 - \frac{\xi_{N,k}}{N}\Big)\Big)\Big)\Big| \\
\leq
\Big( 1 + \frac{|a| \, |z|}{N}\Big)^N \, \exp (|{\rm Im}\, z|)
\leq \exp \big( (|a| + 1)\, |z|\big),
\end{multline*}
which shows that the the family of entire functions
$$
\Big\{z \longmapsto T_{N}^{\boldsymbol \Xi}[a]\, (z)\,:\, a\in K_\R,\
\boldsymbol \Xi\Big\},
$$
where $K_\R$ is any compact subset of $\R$, is a bounded subset of ${\rm Exp} (\C)$.
It remains to prove that, given a compact subset $K_\C$ of the complex plane, the sequence of entire functions \eqref{sect3-prop1bis-eq1} converges uniformly on $K_\C$ towards
$z \mapsto e^{az}$, the convergence being uniform both with respect to $a\in K_\R$ and to the collection $\boldsymbol \Xi$. For any $w$ sufficiently close to $0$ in $\C$, one has
\begin{equation}\label{sect3-prop1bis-eq2}
\begin{split}
\cos w + i a\, \sin w &= 1 + iaw - \frac{w^2}{2} + (|a|+1)\, O(|w|^3) \\
\log (\cos w + ia\, \sin w) & = \log \Big(1 + ia w - \frac{w^2}{2}\Big) + (|a|+1)\, O(|w|^3) \\
& = ia\, w + \frac{a^2-1}{2} \, w^2 (1 - O(|aw|)).
\end{split}
\end{equation}
Then, for $N$ large enough (depending on $K_\C$ and of the compact $K_\R$), one has
\begin{equation*}
\begin{split}
T_{N}^{\boldsymbol \Xi}[a]\, (z) =
\exp \bigg( \sum\limits_{k=1}^N
\Big( ia\, w_{N,k} + \frac{a^2-1}{2}\, w_{N,k}^2 \Big(1 - O \big(|a w_{N,k}|\big)\Big)\Big)\bigg),
\end{split}
\end{equation*}
where $w_{N,k} = (z/N)\times (1 - \xi_{N,k}/N)$ for $k=1,...,N$. Therefore
\begin{equation*}
T_{N}^{\boldsymbol \Xi}[a]\, (z) = \exp (iaz)
\, \exp \Big(-ia z\,  \frac{\sum_{k=1}^N \xi_{N,k}}{N^2}
+ \frac{a^2-1}{2}\, \frac{z^2}{N^2}\,
\sum_{k=1}^N \Big( 1 - \frac{\xi_{N,k}}{N}\Big)^2 \Big(1 - O \big(|a w_{N,k}|\big)\Big)\Big).
\end{equation*}
The uniform convergence of this sequence of functions of $z$ on $K_\C$ towards
$z\mapsto e^{iaz}$ follows immediately. Moreover the convergence is uniform with respect to
$a\in K_\R$ and to $\boldsymbol \Xi$, and this concludes the proof.
\end{proof}

The next result, and the following Example \ref{sect4-expl1}, show  a feature which is somewhat predictable from our previous considerations, namely that restrictions to the real line of entire functions inherit the $({\rm TCSP})_\C$ property with respect to almost regular sampling on $[-1,1]$ subordinate to the approximation of $z \mapsto e^{iaz}$ in ${\rm Exp}(\C)$ locally uniformly with respect to the parameter $a$.

\begin{theorem}\label{sect4-thm2}
Let $\{T_{H,N}[\lambda]\,:\, \lambda\in \R\}_{N\geq 1}$ be a $(\R,H)$ sequence of generalized trigonometric polynomials such that, for any $a\in\mathbb R$, the sequence of entire functions
$$
\Big( z \in \C \longmapsto \sum\limits_{\nu=0}^N C_{H,\nu} (N,a) e^{ih_{N,\nu} z}\Big)_{N\geq 1}
$$
converges towards $z \longmapsto e^{iaz}$ in ${\rm Exp} (\C)$, the convergence being locally uniform with respect to $a\in \R$. Then any restriction $\psi=F_{|\R}$ of an entire function $F$ satisfies the
$({\rm TCSP})_\C$-property with respect to the superoscillating sequence $\{T_{H,N}[\lambda]\,:\, \lambda\in \R\}_{N\geq 1}$. In other terms, the sequence of continuous functions
$$
\Big( (a,a') \in \R^2 \longmapsto \sum\limits_{\nu=0}^N C_{H,\nu} (N,a) \psi(a' + h_{N,\nu})\Big)_{N\geq 1}
$$
converges towards $(a,a') \longmapsto \psi(a+a')$ in $\mathcal C(\R^2,\C)$ locally uniformly over the compact sets of $\R^2$.
\end{theorem}

\begin{proof}
Let $F$ be an entire function and $F(-i d/dz + a')$ the infinite order differential operator
$$
F\Big(-i \frac{d}{dz} + a'\Big)  = \sum\limits_{\kappa = 0}^\infty
\frac{F^{(\kappa)} (0)}{\kappa!} \Big(
-i \frac{d}{dz} + a'\Big)\circ \stackrel{\kappa\ {\rm times}}{\cdots} \circ
\Big(
-i\frac{d}{dz} + a'\Big)=
\sum\limits_{\kappa = 0}^\infty
\frac{F^{(\kappa)} (0)}{\kappa!} \Big(
-i \frac{d}{dz} + a'\Big)^\kappa
$$
where $a'$ denotes a real parameter. Such an operator acts from ${\rm Exp}(\C)$ to itself as follows.
Entire functions with exponential growth are in bijective correspondence with analytic functionals through the Fourier-Borel transform $T \longleftrightarrow \big(\widehat T : z \longmapsto \langle T(\zeta)\,,\, e^{\zeta z}\rangle\big)$.
If $T \in H'(\C)$ is an analytic functional, then the action of the infinite order differential operator
$F(-i d/dz+a')$ on the Fourier-Borel transform of $T$ is given by
\begin{equation}\label{sect4-eq18}
F\Big(-\frac{d}{dz} + a'\Big)
\Big( z \longmapsto \big\langle T(\zeta)\,,\, e^{\zeta z}\big\rangle\Big) =
\Big( z \longmapsto \Big\langle T(\zeta)\,,\, F(-i\zeta + a')\, e^{\zeta z}\Big\rangle\Big).
\end{equation}
Such action is a continuous action from ${\rm Exp}(\C)$ to ${\rm Exp}(\C)$, locally uniformly with respect to the parameter $a'$, see \cite{ACSS18} and \cite[Proposition 2.10]{ACJSSST22} (restricted here to the univariate setting). Since for any $N\in \N^*$, one has that
$$
F\Big(-i \frac{d}{dz} + a'\Big)\Big( z \longmapsto
\sum\limits_{\nu=0}^N C_{H,\nu} (N,a) e^{ih_{N,\nu} z}\Big)
= \Big( z \longmapsto \sum\limits_{\nu=0}^N C_{H,\nu} (N,a)
F(a' + h_{N,\nu})\, e^{i h_{N,\nu} z}\Big),
$$
it follows from the hypothesis on $\{T_{H,N}[\lambda]\,:\, \lambda\in \R\}_{N\geq 1}$ that the sequence of functions
$$
\Big( z \longmapsto \sum\limits_{\nu=0}^N C_{H,\nu} (N,a)
F(a' + h_{N,\nu})\, e^{i h_{N,\nu} z}\Big)_{N\geq 1}
$$
converges in ${\rm Exp}(\C)$ towards
$$
z \longmapsto F\Big(- \frac{d}{dz} + a'\Big) \Big( z \longmapsto e^{iaz}\Big) =
F(a+a') \, e^{iaz},
$$
locally uniformly with respect to $(a,a')\in \R^2$. Such convergence is in particular uniform on the compact subset $\{z=0\}$, locally uniformly with respect to the parameters $(a,a')\in \R^2$.
Theorem \ref{sect4-thm2} is thus proved if one restricts $F$ to the real line thus obtaining $\psi$.
\end{proof}

\begin{example}\label{sect4-expl1} {\rm Theorem {\rm \ref{sect4-thm2}} applies when $\{T_{H,N}[\lambda]\,:\, \lambda\in \R\}_{N\geq 1} =
\{T_{H^{\boldsymbol \epsilon},N}[\lambda]\,:\, \lambda\in \R\}$,
see Proposition 3.1 in \cite{cssy23}, which corresponds to a situation of regular sampling, but also in the case where
$\{T_{H,N}[\lambda]\,:\, \lambda\in \R\}_{N\geq 1} =
\{T_{H^{\boldsymbol \Xi},N}[\lambda]\,:\, \lambda\in \R\}$, see Proposition
 \ref{sect2-prop1bis}, which corresponds to a situation of almost regular sampling.
}
\end{example}

{Let us now consider the general case where the matrix $H$ as in \eqref{sect1-eq0} is such that
$\nu(N)=N$
and the sampling may be  irregular. We will show that Lagrange-Hermite polynomials allow to obtain superoscillations with respect to such sampling (for a particular case see \cite{bcss23}).

\begin{theorem}\label{sect3-prop2}
The sequence of entire functions depending on the real parameter $a$, and defined by
\begin{equation}\label{THN}
\Big( z \longmapsto T_{H,N}^{{\rm Lag}}[a] (z) := \sum\limits_{\nu=0}^N \Big(\prod_{\nu'\not=\nu}
\frac{a - h_{N,\nu'}}{h_{N,\nu} - h_{N,\nu'}}\Big)\, e^{ih_{N,\nu} z}\Big)_{N\geq 1}
\end{equation}
converges in ${\rm Exp}(\C)$ towards $z \mapsto e^{iaz}$, the convergence being locally uniform with respect to the real parameter $a$.
\end{theorem}

\begin{proof}
Let $a\in \R$ and $\Gamma_a : \theta \in [0,2\pi] \longmapsto (|a|+ 2) e^{i\theta}$. It follows from
Hermite remainder theorem, see for example \cite[Theorem 3.6.1]{Davis75}, that for any
$z\in \C$ and $N\in \N^*$,
\begin{equation}\label{sect3-prop2-eq2}
\begin{split}
\Big|e^{iaz} -
\sum\limits_{\nu=0}^N \Big(\prod_{\nu'\not=\nu}
\frac{a - h_{N,\nu'}}{h_{N,\nu} - h_{N,\nu'}}\Big)\, e^{ih_{N,\nu} z}\Big|  & = \Big|
\frac{1}{2i\pi}
\int_{\Gamma_a} \Big(\prod_{\nu=0}^N \frac{a- h_{N,\nu}}{\zeta - h_{N,\nu}}\Big)\, \frac{e^{i\zeta z}}{\zeta -a}\, d\zeta\Big|   \\
& \leq (|a|+2)\, \frac{e^{(|a| + 2)|z|}}{(|a|+2) - |a|}
\Big( \frac{|a| + 1}{(|a|+2)-1}\Big)^{N+1} \\
& = (|a| +2) \, \frac{e^{(|a|+2)|z|}}{2},
\end{split}
\end{equation}
which implies that the sequence of entire functions
$$
\Big(z \longmapsto T_{H,N}^{\rm Lag}[a] (z) = \sum\limits_{\nu=0}^N \Big(\prod_{\nu'\not=\nu}
\frac{a - h_{N,\nu'}}{h_{N,\nu} - h_{N,\nu'}}\Big)\, e^{ih_{N,\nu} z}\Big)_{N\geq 1}
$$
is a bounded sequence in ${\rm Exp}(\C)$, locally uniformly with respect to the real parameter $a$.
On the other hand, it follows from \cite[Theorem 2.1 and Theorem 2.2]{ACSSST21} that for any $a\in \R$ and $N\in \N^*$, the entire functions
$$
z \longmapsto e^{iaz},\quad z \longmapsto T_{H,N}^{\rm Lag}[a] (z) = \sum\limits_{\nu=0}^N \Big(\prod_{\nu'\not=\nu}
\frac{a - h_{N,\nu'}}{h_{N,\nu} - h_{N,\nu'}}\Big)\, e^{ih_{N,\nu} z}
$$
share the same derivatives at the origin up to order less than or equal to $N$, which implies that for any $z\in \C$ and $N\in \N^*$,
$$
\sum_{\kappa = 0}^N \Big( \Big(\frac{d}{dz}\Big)^\kappa  T_{H,N}^{\rm Lag}[a]\Big) (0)\, \frac{z^\kappa}{\kappa!} = \sum\limits_{\kappa =0}^N \frac{(ia)^\kappa}{\kappa!} \, z^\kappa.
$$
 For each such $z\in \C$ and $N\in \N^*$ one has
$$
T_{H,N}^{\rm Lag}[a] (z) - \sum\limits_{\kappa =0}^N \frac{(ia)^\kappa}{\kappa!} \, z^\kappa =
\sum\limits_{\kappa > N} \gamma_{N,\kappa}[a]  z^\kappa
$$
with
$$
|\gamma_{N,\kappa} [a]| \, R^\kappa
\leq (|a|+2) \frac{e^{(|a| +2)R}}{2}
$$
for all $\kappa >N$ and any $R>0$ according to the fact that the sequence \eqref{THN}
is bounded in ${\rm Exp}(\C)$.
Therefore, the sequence of entire functions $\big(\sum_{\kappa >N} \gamma_{N,\kappa}[a] z^\kappa\big)_{N\geq 1}$ converges towards $0$ on any compact subset of $\C$, the convergence being locally uniform with respect to the real parameter $a$. It follows that the sequence \eqref{THN} converges uniformly on any compact of $\C$ towards $z \mapsto e^{iaz}$, the convergence being locally uniform with respect to the real parameter $a$. This concludes the proof.
\end{proof}

As a consequence we have the following result

\begin{theorem}\label{Thm27}
Any restriction $\psi: a\in \R \mapsto \psi_a=\psi(a) \in \C$ of an entire function to the real line
inherits the $({\rm TCSP})_\C$ property on $\R$ with respect to any superoscillating $(\R,H)$ superoscillating sequence $\{x \mapsto T_{H,N}^{\rm Lag}[\lambda] (x)\,;\, \lambda\in \mathbb R\}_{N\geq 1}$ as in (2.2), where $H$ is as in (1.1) with $\nu(N)=N$ for any $N\in \mathbb N^*$.
\end{theorem}
\begin{proof}
It follows from the combination of Theorem 2.6 with the argument employed to prove Theorem 2.4.
\end{proof}

Let now $A$ be any open interval of $\R$ which contains $[-1,1]$.
Consider the case where the superoscillating sequence $\{T_{H,N}[\lambda]\,:\, \lambda\in \R\}_{N\geq 1}$ is attached to an irregular sampling on $[-1,1]$, as before.
Observe that continuous real valued functions on $A$ fail in general to satisfy $({\rm SP})_\C$  with respect to $\{T_{H,N}^{\rm Lag}[\lambda]\,:\, \lambda\in \R\}_{N\geq 1}$ since there is a continuous function $\psi_H : A \rightarrow \R$ which for which
\begin{equation}\label{sect3-final-eq1}
{\rm mes}
\Big\{a \in [-1,1]\,:\, \limsup\limits_{N\rightarrow +\infty} \Big| \psi_H (a) - \sum_{\nu=0}^N \Big(\prod_{\nu'\not=\nu}
\frac{a - h_{N,\nu'}}{h_{N,\nu} - h_{N,\nu'}}\Big)\, \psi_H(h_{N,\nu})\Big| < \infty\Big\}
= 0,
\end{equation}
see \cite{ErdV80}. Even when the sampling is regular, namely $h_{N,\nu} = 1-2\nu/N$, such is the
case when $\psi_H : a \mapsto |a|$ since one has in this case
\begin{equation}\label{sect3-final-eq2}
\forall\, a \in [-1,1]\setminus \{0\},\
\limsup\limits_{N\rightarrow +\infty} \Big| |a| - N^{N} \sum_{\nu=0}^N
\Big(\prod_{\nu'\not=\nu}
\frac{a - (1-2\nu'/N)}{2(\nu'-\nu)}\Big)\, \big|1-2\nu/N|\Big| = +\infty,
\end{equation}
as proved in \cite{Bern18}, see also \cite{Rev2000} for updated results and references which illustrate how poorly the Lagrange-Hermite interpolation behaves under only the hypothesis of continuity, if one compares with properties $({\rm SP})_\C$ or $({\rm TCSP})_\C$ attached to superoscillating sequences
$\{T_{N}^{\boldsymbol \epsilon}[a]\,:\, a\in \R\}$ subordinate to regular sampling as in \cite{cssy23}.
}

The next proposition, together with subsequent Remark \ref{sect4-rem2}, shows how a well-known example of Kantorovich can be used to construct a continuous function $\psi~: \R\rightarrow \C$ which satisfies
$({\rm SP})_\C$, provided such notion is slightly modified according to \cite[Definition 4.7]{ACJSSST22}.
In particular, it indicates that, if one tolerates irregular sampling, the extrapolation (from their irregularly sampled values on $[-1,1]$) of some particular classes of functions, continuous but non-real analytic,
 through a sequence or real-analytic functions which
converges uniformly on any compact subset of $\mathbb R$ is indeed possible.

To explicitly construct such functions we recall the following result (see Theorem 3.7 in \cite{cssy23}).

\begin{theorem}\label{sect2new-thm2}
Let $G^-$ and $G^+$ two entire functions such that
\begin{equation}\label{sect2new-eq11}
G^{-} (1/2) = G^+(1/2),\quad \frac{dG^-}{dz}(1/2) \not= \frac{dG^+}{dz}(1/2)
\end{equation}
and let $g : b \in \R \rightarrow \C$ be the continuous map defined by
\begin{equation}\label{sect2new-eq12}
g (b) = \begin{cases} G^-(b)\ & {\rm if}\ b < 1/2 \\
G^{+} (b)\ & {\rm if}\ b \geq 1/2.
\end{cases}
\end{equation}
There exists an open interval $A_0=]-1-\rho_0,1+\rho_0[\subset \R$ containing $[-1,1]$ such that the function
$$a \in A_0 \longmapsto \psi_0(a) = g ((1+a)/2)$$
is a regular $\C$-supershift in $A_0$.
As a consequence, the fact that $\psi=\psi_0*\theta$, for $\theta$ a regularizing test-function, is a smooth regular $\C$-supershift on some open interval
$A\subset \R$ with diameter strictly larger than $2$ does not imply in general that $\psi$ is real analytic on $A$.
\end{theorem}

With the notations in this theorem we now have:

\begin{proposition}\label{sect4-prop2}
Let $A_0 = ]-1-\rho_0,1+\rho_0[$ and $\psi_0: A_0 \rightarrow \C$ be the smooth non-real analytic function defined in Theorem  \ref{sect2new-thm2}.
Let
\begin{equation}\label{sect4-eq15}
\Theta_{\rho_0} : a \in \R \longmapsto \alpha = \frac{2 (1+ \rho_0)}{\pi}
{\rm arctan}\, \Big( a\,
{\rm tan}\, \frac{\pi}{2(1+\rho_0)}\Big)
\end{equation}
and, for any $\boldsymbol \epsilon \in [0,1[^{\N^*}$ such that
$\lim_{N\rightarrow +\infty} \epsilon_N =0$,
let the matrix $H_{\rho_0}^{\boldsymbol \epsilon}$ as in
\eqref{sect1-eq0} be defined as follows: for any $N\in \N^*$ and $0\leq \nu\leq \nu(N)=N$,
\begin{equation}\label{sect4-eq16}
\begin{split}
h_{\rho_0, N,\nu}^{\boldsymbol \epsilon} & = \Theta_{\rho_0}^{-1}
\Big( 1 - 2\, \frac{\nu + \epsilon_N (N-\nu)}{N}\Big) \\
& = \frac{2 (1+\rho_0)}{\pi}
\, {\rm arctan}\, \Big(\Big( 1 - 2\, \frac{\nu + \epsilon_N (N-\nu)}{N}\Big)\,
{\rm tan}\, \frac{\pi}{2(1+\rho_0)}\Big) \in [-1,1].
\end{split}
\end{equation}
Then, if $\psi := \psi_0\, \circ\, \Theta_{\rho_0} \in \mathcal C(\R,\C)$,
$\psi$ is a non-real analytic continuous function such that the sequence of real-analytic functions
\begin{equation}\label{sect4-eq17}
\Big(a \in \R \longmapsto
\sum\limits_{\nu=0}^N
\binom{N}{\nu}
\Big( \frac{1 + \Theta_{\rho_0}(a)}{2}\Big)^{N-\nu}
\Big( \frac{1 -\Theta_{\rho_0}(a)}{2}\Big)^{\nu} \, \psi (h_{\rho_0,N,\nu}^{\boldsymbol \epsilon})\Big)_{N\geq 1}
\end{equation}
converges in $\mathcal C(\R,\R)$ uniformly over the compact sets towards $a \longmapsto \psi(a)$, thus extrapolates $\psi$ asymptotically uniformly on any compact subset of $\R$ from its values in $[-1,1]$. Moreover, the convergence on each compact subset of $\R$ is uniform with respect to $\boldsymbol \epsilon$, provided
$\boldsymbol \epsilon$ remains in a family $\{\boldsymbol \epsilon_{\iota'}\in [0,1[^{\N^*}\,:\, \iota'\in I'\}$ which elements tend to $0$ at infinity uniformly with respect to the index $\iota'$.
\end{proposition}

\begin{proof} For any $a\in \R$ and $N\in \N^*$, one has
\begin{multline*}
\sum\limits_{\nu=0}^N
\binom{N}{\nu}
\Big( \frac{1 + \Theta_{\rho_0}(a)}{2}\Big)^{N-\nu}
\Big( \frac{1 -\Theta_{\rho_0}(a)}{2}\Big)^{\nu} \, \psi (h_{\rho_0,N,\nu}^{\boldsymbol \epsilon}) \\
= \sum\limits_{\nu=0}^N
\binom{N}{\nu}
\Big( \frac{1 + \Theta_{\rho_0}(a)}{2}\Big)^{N-\nu}
\Big( \frac{1 -\Theta_{\rho_0}(a)}{2}\Big)^{\nu} \, \psi_0 \big( \Theta_{\rho_0} (h_{\rho_0,N,\nu}^{\boldsymbol \epsilon})\big) \\
= \sum\limits_{\nu=0}^N
\binom{N}{\nu}
\Big( \frac{1 + \Theta_{\rho_0}(a)}{2}\Big)^{N-\nu}
\Big( \frac{1 -\Theta_{\rho_0}(a)}{2}\Big)^{\nu} \, \psi_0 \Big( 1 - 2\, \frac{\nu + \epsilon_N (N-\nu)}{N}\Big).
\end{multline*}
The result follows then immediately from Theorem \ref{sect2new-thm2}, which asserts that
$\psi_0$ is a regular $\C$-supershift on $A_0$, which is homeomorphic to $\R$ through the real-analytic homeomorphism $\Theta_{\rho_0}$ (which restriction to $[0,1]$ realizes a continuous automorphism of $[-1,1]$). The non-real analyticity of $\Psi$ follows from the non-real analyticity of $\psi_0$, together with the real analyticity of $\Theta_{\rho_0}$.
\end{proof}

\begin{remark}\label{sect4-rem2} If one considers instead of the function $\psi_0$ the restriction to $A_0$ of the entire function
$e^{i(\cdot) x}: z \in \C \mapsto e^{i z\, x}$, where $x\in \R$, the same proof than that of Proposition 3.2 in \cite{cssy23} shows that the sequence of real analytic functions
$$
\Big( a \in \R \longmapsto \sum\limits_{\nu=0}^N
\binom{N}{\nu}
\Big( \frac{1 + \Theta_{\rho_0}(a)}{2}\Big)^{N-\nu}
\Big( \frac{1 -\Theta_{\rho_0}(a)}{2}\Big)^{\nu} \, \big(\exp
\big( i\, (\cdot)\, x\big) \circ \Theta_{\rho_0}\big)
(h_{\rho_0,N,\nu}^{\boldsymbol \epsilon})\Big)_{N\geq 1}
$$
converges in $\mathcal C(\R,\C)$ uniformly over the compact sets towards $a \longmapsto \exp \big(i\, \Theta_{\rho_0} (a)\, x\big)$, the convergence being uniform both with respect to the parameter $x$ provided $x$ remains in a compact subset of $\R$ and to the sequence $\boldsymbol \epsilon$ provided $\boldsymbol \epsilon$ remains in a family $\{\boldsymbol \epsilon_{\iota'}\in [0,1[^{\N^*}\,:\, \iota'\in I'\}$ which elements tend to $0$ at infinity uniformly with respect to the index $\iota'$. The $(\R, H^{\boldsymbol \epsilon})$-sequence of generalized trigonometric polynomials
$$
\Bigg\{
x \longmapsto
\sum\limits_{\nu=0}^N
\binom{N}{\nu}
\Big( \frac{1 + \Theta_{\rho_0}(\lambda)}{2}\Big)^{N-\nu}
\Big( \frac{1 -\Theta_{\rho_0}(\lambda)}{2}\Big)^{\nu} \, \big(\exp
\big( i\, (\cdot)\, x\big) \circ \Theta_{\rho_0}\big)
(h_{\rho_0,N,\nu}^{\boldsymbol \epsilon})
\,:\, \lambda \in \R \Bigg\},
$$
is then a superoscillating $(\R, H^{\boldsymbol \epsilon}
)$-sequence subordinate to the pair of functions
$$
g_{\rho_0} : \lambda \in \R \longmapsto \Theta_{\rho_0} (\lambda),\quad
C_{\rho_0} : \lambda \in \R \longmapsto 1.
$$
Therefore the continuous non-real analytic function $\psi$ constructed in Proposition
{\rm \ref{sect4-prop2}} inherits the property $({\rm SP})_\C$ with respect to all
$(\R,H_{\rho_0}^{\boldsymbol \epsilon})$ for any $\boldsymbol \epsilon$, provided the definition
of $({\rm SP})_\C$ is given in similar terms than in {\rm \cite[Definition 4.7]{ACJSSST22}} (we restrict ourselves here to the univariate situation).
\end{remark}

\section{Regular $\C$-supershifts on unbounded intervals, periodic regular $\C$-supershifts}\label{sect4}

Let $A$ be an unbounded interval and let $\psi~: A \rightarrow \C$ be a continuous regular $\C$-supershift on $A$. We have the following result:
\begin{lemma}\label{sect4-lem1}
Let $A$ be an interval of $\R$ and $a_0\in \R$. The fact that $\psi$ is a continuous
regular $\C$-supershift on $A$ is equivalent to the fact that $a \mapsto \psi(a-a_0)$ is a continuous regular $\C$-supershift on
the shifted interval $a_0 + A = \{a_0 + a\,:\, a\in A\}$, or that $a \in -A \longmapsto \psi(-a)$ is a continuous regular $\C$-supershift on the symmetric interval $-A$.
\end{lemma}

\begin{proof} The fact that $\psi$ is a continuous regular $\C$-supershift on $A$ is equivalent to the fact that $a \mapsto \psi (a-a_0)$ is a continuous regular $\C$-supershift on $a_0 + A$ because $({\rm TCSP})_\C$ commutes with translations. To prove the second part of the statement, let $\psi$ be a continuous $\C$-regular supershift in $A$.
Then, for any $(a,a') \in - \mathbb A
= \{(a,a') \in \R\times -A\,:\, a'+ [-1,1] \subset -A,\ a+a' \in -A\}$, for any $\boldsymbol \epsilon
\in [0,1[^{\N^*}$ with $\lim_{N\rightarrow +\infty} \epsilon_N =0$ and any
$N\in \N^*$,
\begin{equation}\label{sect4-eq1}
\begin{split}
\sum\limits_{\nu=0}^N
\binom{N}{\nu}&
\Big(\frac{1+a}{2}\Big)^{N-\nu}
\, \Big(\frac{1-a}{2}\Big)^{\nu} \psi\Big(- \Big( a' + 1 -
2\, \Big( \frac{\nu + \epsilon_N (N-\nu)}{N}\Big)\Big) \Big) \\
&= \sum\limits_{\nu=0}^N
\binom{N}{\nu}
\Big(\frac{1-a}{2}\Big)^{N-\nu}
\Big(\frac{1+a}{2}\Big)^{\nu} \psi\Big(- \Big( a' + 1 - 2
\Big( \frac{N-\nu + \epsilon_N \nu}{N}\Big)\Big) \Big) \\
&= \sum\limits_{\nu=0}^N
\binom{N}{\nu}
\Big(\frac{1-a}{2}\Big)^{N-\nu}
\Big(\frac{1+a}{2}\Big)^{\nu}
\psi \Big(- a' + 1 - 2\Big (\frac{1-\epsilon_N}{N}\Big)\, \nu\Big)  \\
&= \sum\limits_{\nu=0}^N
\binom{N}{\nu}
\Big(\frac{1-a}{2}\Big)^{N-\nu}
\Big(\frac{1+a}{2}\Big)^{\nu}
\psi \Big(- a' + 2 \epsilon_N + 1 - 2\, \Big(\frac{\nu + \epsilon_N (N-\nu) }{N}\Big)\Big).
\end{split}
\end{equation}
Since $\psi$ is a continuous regular $\C$-supershift on $A$, one has, if
$\{\boldsymbol \epsilon_{\iota'} = (\epsilon_{\iota',N})_{N\geq 1}\,:\, \iota'\in I'\}$
is a family of sequences which all tend to $0$ at infinity uniformly with respect to the index $\iota$, that all sequences of functions
$$
\Big( \sum\limits_{\nu=0}^N
\binom{N}{\nu}
\Big(\frac{1-a}{2}\Big)^{N-\nu}
\Big(\frac{1+a}{2}\Big)^{\nu}
\psi \Big(- a' + 2\, \epsilon_{\iota',N} + 1 - 2\, \Big(\frac{\nu + \epsilon_{\iota',N} (N-\nu) }{N}\Big)\Big)\Big)_{N\geq 1},
$$
for $(a,a') \in -\mathbb A$
and $\iota'\in I'$, converge uniformly on any compact subset of $-\mathbb A$ towards $(a,a') \longmapsto \psi(-a-a')$, the convergence being moreover uniform with respect to the index $\iota'$.
It follows from the equalities \eqref{sect4-eq1} that $a \mapsto \psi(-a)$ is a regular $\C$-supershift on $-A$, which concludes the proof.
\end{proof}
According to Lemma 3.1, we will now restrict our attention to one of the following two cases: $A=\R$ or $A = \R_{>-1}$, where $\R_{>t} = (t,+\infty)$. Let $B= \Upsilon^{-1} (A)$, where $\Upsilon : t \longmapsto 2t-1$, that is $B = \R$ or $B= \R_{>0}$, so that
$\mathbb B = \{(b,b') \in \R \times B\,:\, b'+[0,1] \subset B, b+b'\in B\}$ equals either $\R^2$ or $\{(b,b')\in \R \times \R_{>0}\,:\, b+b'>0\}$. Let $\psi$ be a continuous function on
$A$ and $\Psi = \psi \circ \Upsilon : B \rightarrow \R$. Let $b'\in B$.
Given $\boldsymbol \epsilon \in [0,1[^{\N^*}$ with $\lim_{N \rightarrow +\infty} \epsilon_N =0$, let, for each $N\in \N^*$,
\begin{equation}\label{sect3-eq1bis}
\mathbb T^{\boldsymbol \epsilon}_N(\Psi)(X,b') = \sum\limits_{\nu=0}^N
\binom{N}{\nu} X^\nu(1-X)^{N-\nu} \Psi \Big( b' + \nu \frac{1-\epsilon_N}{N}\Big) \in \C[X].
\end{equation}
It follows from \cite[Proposition 3.3]{cssy23} that
\begin{equation}\label{sect3-eq2bis}
\begin{split}
\mathbb T^{\boldsymbol \epsilon}_N(\Psi) (X,b') &= \sum\limits_{\kappa=0}^N
\frac{N!}{(N-\kappa)!} \,
\frac{(\Delta^+_{(1-\epsilon_N)/N})^\kappa [\Psi] (b')}{\kappa!}\, X^\kappa \\
& = \sum\limits_{\kappa =0}^N \Big(\prod_{0\leq j<\kappa} \Big(1 - \frac{j}{N}\Big)\Big) \,
\frac{\big(\big(N/(1-\epsilon_N)\big)\, \Delta^+_{(1-\epsilon_N)/N}\big)^\kappa[\Psi](b')} {\kappa!}\, \big((1-\epsilon_N) X\big)^\kappa \\
& = \sum\limits_{\kappa = 0}^N a_{N,\kappa}^{\boldsymbol \epsilon} (\Psi;b') \, X^\kappa,
\end{split}
\end{equation}
where
$$
\big(\Delta^+_{(1-\epsilon_N)/N}\big)^\kappa \quad (\kappa \in \N)
$$
denote the successive forward (+) differences of $\Psi$ evaluated at $b'$ with sampling rate $(1-\epsilon_N)/N$ and $\prod_{j\in \emptyset} (1-j/N) = 1$.

\begin{remark}\label{sect3-rem1bis}
{\rm Suppose that $\psi$ is $C^\infty$ on $A$, which amounts to say that $\Psi$ is
$C^\infty$ on $B$. Then, for any $\boldsymbol \epsilon \in[0,1[^{\N^*}$ with
$\lim_{N\rightarrow +\infty} \epsilon_N =0$, any $b'\in B$, $\kappa \in \N$ and $N\geq \kappa$,
\begin{multline}\label{sect3-eq3bis}
\frac{\big(\big(N/(1-\epsilon_N)\big)\, \Delta^+_{(1-\epsilon_N)/N}\big)^\kappa[\Psi](b')} {\kappa!} =
\int_{\tau \in \boldsymbol \Delta_\kappa} \Psi^{(\kappa)} \Big( b' + \frac{1-\epsilon_N}{N} (\tau_0 +
\cdots + \tau_{\kappa-1})\Big)\, d\tau \\
= \frac{1}{\kappa!}\, \Big({\rm Re}\, (\Psi^{(\kappa)}) \Big( b' + \kappa \frac{1-\epsilon_N}{N}\, \xi^{\boldsymbol \epsilon}
_{N,\kappa}(b')\Big) + i\, {\rm Im}\, (\Psi^{(\kappa)}) \Big( b' + \kappa \frac{1-\epsilon_N}{N}\,\eta^{\boldsymbol \epsilon}
_{N,\kappa}(b')\Big) \Big)
\end{multline}
according to Rolle's theorem, see for example \cite[\textsection 1.3 (1.80)]{Phil03}, where
$\boldsymbol \Delta_k$ denotes the elementary simplex $\{(\tau_0,..,\tau_{\kappa-1}) \in \R^\kappa\,:\, 0 \leq \tau_0\leq \cdots \leq
\tau_{\kappa-1} \leq 1\}$ (or $\{0\}$ when  $\kappa =0$)
with euclidean volume $\int_{\boldsymbol \Delta_\kappa} d\tau = 1/\kappa!$ (or atomic mass $|d\tau|= 1$ when $\kappa =0$) and
$\xi^{\boldsymbol \epsilon}_{N,\kappa}(b'), \eta^{\boldsymbol \epsilon}_{N,\kappa}(b')\in [0,1]$.  As a consequence, one can interpret (assuming only the sole continuity of $\psi$ on $A$) each polynomial
$\mathbb T^{\boldsymbol \epsilon}_N(\Psi)(X,b')\in \C[X]$ defined in \eqref{sect3-eq2bis}, where
$\boldsymbol \epsilon \in [0,1[^{\N^*}$ is such that $\lim_{N\rightarrow +\infty} \epsilon_N =0$, as a numerical substitute (in terms of discrete differential calculus) for the Principal Part at order $N$ of the Taylor series of $\Psi$ about $b'$.
}
\end{remark}
Since the concept of Taylor series of $\Psi$ at $b'\in B$ does not make sense for continuous functions, Remark \ref{sect3-rem1bis} suggests to introduce, for each $M\in \N^*$ and $\boldsymbol \epsilon \in [0,1[^{\N^*}$ such that
$\lim_{N\rightarrow +\infty} \epsilon_N =0$,
the {\it $(\boldsymbol \epsilon,1/M)$-numerical Taylor series}
\begin{equation}
\mathbb S^{\boldsymbol \epsilon}_{1/M}(\Psi)(X,b') := \sum\limits_{\kappa =0}^\infty
a_{M\kappa,\kappa}^{\boldsymbol \epsilon} (\Psi;b') X^\kappa \in \C[[X]]
\end{equation}
of
$\Psi$ at $b'\in B$. Here $M\in \N^*$ and $1/M$ is interpreted as a numerical precision.
Finally, in order to describe the behavior of the sequence of continuous functions
$$
\Big( (b,b') \in \mathbb B \longmapsto \mathbb T^{\boldsymbol \epsilon}_N(\Psi) (b,b')\Big)_{N\in \N^*},
$$
where $\boldsymbol \epsilon \in [0,1[^{\N^*}$ is such that $\lim_{N\rightarrow +\infty}
\epsilon_N =0$, it is convenient to introduce the orthonormal basis $\{\ell_\nu\,:\, \nu \in \N\}$ of
$L^2([0,1],\C,d\xi)$ given by the shifted Legendre polynomial functions
\begin{equation}\label{sect3-eq4bis}
\xi \in [0,1] \longmapsto \ell_\nu (\xi) = \sqrt{2\nu +1} \, L_\nu(1-2\xi) = \sqrt{2\nu +1} \, \sum\limits_{\kappa=0}^\nu (-1)^\kappa
\binom{\nu}{\kappa} \binom{\nu + \kappa}{\kappa} \xi^\kappa,\quad \nu \in \N.
\end{equation}
What suggests the introduction of the orthonormal basis of shifted Legendre polynomials here is that any holomorphic function $G$ in the interior $E_0$ of an ellipse with foci $\{0,1\}$ which sum of semi-axis equals $R_0>0$ admits in $E_0$ (in particular in $[0,1]$) a Neumann expansion $G= \sum_{\nu=0}^\infty \gamma_\nu \ell_\nu$ with $\limsup_{\nu \rightarrow +\infty} |\gamma_\nu|^{1/\nu} \leq 1/R_0$
\cite{Neum1862}. For each $R \geq 0$, $b'\in B$, $N\in \N^*$ and $0\leq \nu \leq N$, let
\begin{equation}\label{sect3-eq5bis}
\gamma^{\boldsymbol \epsilon}_{N,\nu}(\Psi) (R,b') = \int_0^1 \mathbb T^{\boldsymbol \epsilon}_N(\Psi) (R\xi,b')\,
\ell_\nu (\xi)\, d\xi,
\end{equation}
so that one has on $L^2([0,1],\C,d\xi)$ the orthonormal decomposition
\begin{equation} \label{sect3-eq6bis}
\mathbb T_N^{\boldsymbol \epsilon}(\Psi) (R\xi,b') =
\sum\limits_{\nu=0}^N \gamma^{\boldsymbol \epsilon}_{N,\nu}(\Psi) (R,b')\, \ell_\nu(\xi)
\end{equation}
with $\|\mathbb T^{\boldsymbol \epsilon}_N (\Psi)(R\xi,b')\|_{L^2([0,1],\C,d\xi)}^2 =
\sum_{\nu=0}^N |\gamma^{\boldsymbol \epsilon}_{N,\nu}(\Psi) (R,b')|^2$.
It follows from the explicitation of the shifted Legendre polynomials $\ell_\nu$ as in
\eqref{sect3-eq4bis} that one can rewrite \eqref{sect3-eq6bis} as
\begin{equation}
\label{sect3-eq7bis}
\begin{split}
\mathbb T_N^{\boldsymbol \epsilon}(\Psi) (R\xi,b') & =
\sum\limits_{\nu=0}^N \gamma^{\boldsymbol \epsilon}_{N,\nu}(\Psi) (R,b') \Big(
\sqrt{2\nu +1} \, \sum\limits_{\kappa=0}^\nu (-1)^\kappa
\binom{\nu}{\kappa} \binom{\nu + \kappa}{\kappa} \xi^\kappa\Big) \\
&= \sum\limits_{\kappa =0}^N (-1)^\kappa \Big(
\sum\limits_{\nu=\kappa}^N \sqrt{2\nu +1} \binom{\nu}{\kappa} \binom{\nu + \kappa}{\kappa}
\gamma^{\boldsymbol \epsilon}_{N,\nu}(\Psi) (R,b')\Big) \xi^\kappa \\
& = \sum_{\kappa = 0}^N (-1)^\kappa \binom{2\kappa}{\kappa}
\Big( \sum\limits_{\nu=\kappa}^N \sqrt{2\nu+1}
\binom{\nu+ \kappa}{\nu-\kappa} \gamma^{\boldsymbol \epsilon}_{N,\nu}(\Psi) (R,b')\Big)
\xi^\kappa,
\end{split}
\end{equation}
which leads by identification with \eqref{sect3-eq2bis} to the set of relations
\begin{equation}\label{sect3-eq8bis}
a_{N,\kappa}^{\boldsymbol \epsilon} (\Psi;b')\, R^\kappa = (-1)^\kappa \binom{2\kappa}{\kappa} \sum\limits_{\nu=\kappa}^N \sqrt{2\nu+1}
\binom{\nu+ \kappa}{\nu-\kappa} \gamma^{\boldsymbol \epsilon}_{N,\nu}(\Psi) (R,b')
\end{equation}
for any $N\in \N^*$, $0\leq \kappa \leq N$, $b'\in B$, $R\geq 0$ and $\boldsymbol \epsilon \in [0,1[^{\N^*}$ such that $\lim_{N\rightarrow +\infty} \epsilon_N =0$. One can prove the following result.

\begin{proposition}\label{sect3-prop1bis}
Let $A,B,\psi,\Psi$ as above. Suppose in addition that $\psi$ is a $\C$-regular supershift on $A$.
Then, for any $b'\in B$, for any $M\in \N^*$ and $\boldsymbol \epsilon \in [0,1[^{\N^*}$ with $\lim_{N\rightarrow +\infty} \epsilon_N =0$,
\begin{equation}\label{sect3-eq9bis}
\limsup_{\kappa \rightarrow +\infty} \big(a^{\boldsymbol \epsilon}_{M\kappa,\kappa} (\Psi;b')\big)^{1/\kappa} = 0.
\end{equation}
Moreover, for any $M\in \N^*$ and any family $\{\boldsymbol \epsilon_{\iota'}\,:\, \iota'\in I'\}$ of sequences $\boldsymbol \epsilon_{\iota'}$ which all converge towards $0$, the convergence being uniform with respect to the index $\iota'$, the family of continuous functions
$$
\{(z,b') \in \C \times B \longmapsto \mathbb S^{\boldsymbol \epsilon_{\iota'}}_{1/M}(\Psi) (z,b')\,:\, \iota'\in I'\}
$$
(holomorphic in $z$) is a bounded family on any compact subset of $\C\times B$.
\end{proposition}

\begin{proof} Since $\psi$ is a $\C$-regular supershift on $A$, then, for any
$\boldsymbol \epsilon \in [0,1[^{\N^*}$ such that $\epsilon_N \to 0$ for $N\rightarrow +\infty$,
the sequence of continuous functions
$$
\big( (b,b') \in \mathbb B \longmapsto \mathbb T^{\boldsymbol \epsilon}_N (\Psi) (b,b')\big)_{N\geq 1}
$$
converges uniformly towards $(b,b') \longrightarrow \Psi(b+b')$ on any compact subset of $\mathbb B$.
Moreover, given a family $\{\boldsymbol \epsilon_{\iota'}\,:\, \iota'\in I'\}$ of sequences in $[0,1[^{\N^*}$ which all converge to $0$, the convergence being uniform with respect to the index $\iota'\in I'$, the convergence towards $(b,b') \longmapsto \Psi(b+b')$
of the family of sequences of continuous functions
$$
\big\{ \big( (b,b') \in \mathbb B \longmapsto \mathbb T^{\boldsymbol \epsilon_{\iota'}}_N (\Psi) (b,b')\big)_{N\geq 1}\,:\, \iota'\in I'\big\}
$$
is uniform with respect to the index $\iota'\in I'$. As a consequence, for each $b'\in B$, $R\geq 0$ and
$\boldsymbol \epsilon \in [0,1[^{\N^*}$ with $\lim_{N\rightarrow +\infty} \epsilon_N =0$, the sequence
$$
\Big( \big( \gamma^{\boldsymbol \epsilon}_{N,\nu}(\Psi) (R,b')\big)_{\nu \in \N}\Big)_{N\geq 1},
$$
where $\gamma^{\boldsymbol \epsilon}_{N,\nu}(\Psi) (R,b')$ is defined
by \eqref{sect3-eq5bis} when $\nu\leq N$ and by $0$
when $\nu>N$, converges in $\ell^2(\N)$ towards
\begin{equation}\label{sect3-eq10bis}
\Big(
\int_0^1 \Psi(R\xi + b')\,  \ell_\nu (\xi)\, d\xi\Big)_{\nu \in \N}.
\end{equation}
Moreover, given a family $\{\boldsymbol \epsilon_{\iota'}\,:\, \iota'\in I'\}$ of sequences in $[0,1[^{\N^*}$ which all converge to $0$, the convergence being uniform with respect to the index $\iota'\in I'$, and a compact subset $K$ of $B$, the convergence in $\ell^2(\N)$
towards \eqref{sect3-eq10bis} of the family of sequences
$$
\Big\{
\Big( \big( \gamma^{\boldsymbol \epsilon_{\iota'}}_{N,\nu}(\Psi) (R,b')\big)_{\nu \in \N}\Big)_{N\geq 1}\,:\, \iota'\in I',\ b'\in K\Big\}
$$
is uniform with respect to both $b'\in K$ and the index $\iota'\in I'$. Fix now $M\in \N^*$,
$b'\in B$ and $R>0$ and $\boldsymbol \epsilon \in [0,1[^{\N^*}$ such that
$\lim_{N\rightarrow +\infty} \epsilon_N=0$.
It follows from the previous considerations, together with relations \eqref{sect3-eq8bis} with $N=M\kappa$, Cauchy criterion in $\ell^2(\N)$ and Cauchy-Schwarz inequality, that for
$\kappa$ in $\N$ large enough (depending on $M$, $b'$ and $R$)
\begin{equation}\label{sect3-eq11bis}
\begin{split}
|a_{M\kappa,\kappa}^{\boldsymbol \epsilon} (\Psi;b')|\,
\Big( \frac{e(M+1)}{2}\Big)^3 R\Big)^\kappa &\leq \sqrt \kappa \binom{2\kappa}{\kappa}
\binom{(M+1)\, \kappa}{(M-1)\, \kappa} \sqrt{M(2M\kappa +1)}.
\end{split}
\end{equation}
It follows from Stirling's formula (when $M\geq 2$, hence $(M-1)\kappa \geq \kappa$) that, provided the choice of $\kappa$ large enough is updated,
\begin{equation}\label{sect3-eq12bis}
\begin{split}
\binom{(M+1)\, \kappa}{(M-1)\,\kappa} & =
\frac{1}{(2\kappa)!}
\frac{((M+1)\kappa)!}
{((M-1)\, \kappa)!} \leq \frac{3}{2} \, \frac{e^{-2\kappa}}{(2\kappa)!}\,  \sqrt{\frac{M+1}{M-1}}
\frac{((M+1)\kappa)^{(M+1)\kappa}}{((M-1)\kappa)^{(M-1)\kappa}} \\
& = \frac{3}{2}\, e^{-2\kappa} \, \frac{(2\kappa)^{2\kappa}}{(2\kappa)!} \,
\Big(\frac{M+1}{M-1}\Big)^{(M-1)\kappa +1/2} \, \Big(
\frac{M+1}{2}\Big)^{2\kappa} \\
& \leq \frac{3}{2} \, e^{-2\kappa}\, \frac{ (2\kappa)^{2\kappa}}{(2\kappa)!} \,
\Big(\Big( 1 + \frac{2}{M-1}\Big)^{M-1}\Big)^\kappa\,
\Big(
\frac{M+1}{2}\Big)^{2\kappa} \, \sqrt{M+1} \\
& \leq  \frac{3}{2} \, e^{-2\kappa}\, \frac{ (2\kappa)^{2\kappa}}{(2\kappa)!}\, \Big(
e\, \frac{M+1}{2}\Big)^{2\kappa} \, \sqrt{M+1} \\
& \leq \frac{1}{\sqrt{\pi \kappa}}\, \Big(
e\, \frac{M+1}{2}\Big)^{2\kappa} \, \sqrt{M+1}.
\end{split}
\end{equation}
Substituting \eqref{sect3-eq12bis} in the right-hand side of \eqref{sect3-eq11bis} leads to
\begin{equation*}
|a_{M\kappa,\kappa}^{\boldsymbol \epsilon} (\Psi;b')|\,
\Big( \frac{e(M+1)}{2}\Big)^3 R\Big)^\kappa \leq \sqrt{\frac{2\kappa}{\pi}}\,
\Big( e \frac{(M+1)}{2}\Big)^{2\kappa} (M+1) \leq \sqrt{\frac{2\kappa}{\pi}}\Big( e \frac{(M+1)}{2}\Big)^{3\kappa},
\end{equation*}
hence to
\begin{equation}\label{sect3-eq13bis}
|a_{M\kappa,\kappa}^{\boldsymbol \epsilon} (\Psi;b')| \, R^\kappa \leq \sqrt{\frac{2\kappa}{\pi}}
\end{equation}
for $\kappa$ large enough, which proves the first assertion of the proposition. The second assertion assertion follows from the fact that estimates \eqref{sect3-eq13bis} remain valid when $\boldsymbol \epsilon$ is replaced by $\boldsymbol \epsilon_{\iota'}$ and $b'$ by an arbitrary element of a compact subset $K$ of $B$, still for $\kappa$ large enough (independently of $\iota'
\in I'$ and $b'\in K$).
\end{proof}

Proposition \ref{sect3-prop1bis}, combined with Remark \ref{sect3-rem1bis}, suggest naturally the following conjecture, where the regularity of the sampling has a crucial role.

\begin{problem}\label{sect4-conj1}
A $\C$-valued continuous function $\psi$ on the real line is a $\C$-regular supershift if and only if $\psi$ is the restriction to the real line of an entire function.
\end{problem}


\begin{remark}
If true, such conjecture would imply that for some class of potentials $x \in U \longmapsto V(x)$ (for example $U=\R$ and $V$ smooth, real, $2\pi$-periodic and non-constant), the evolution $a \mapsto \psi_a(t,x)$ through the Cauchy problem for the Schr\"odinger equation would be such that, for some
$\varphi \in \mathcal D( (0,T) \times U,\C)$, the continuous map $a\mapsto
\langle \psi_a(t,x),\varphi(t,x)\rangle$ fails to inherit either Property (1) or (2) in Definition \ref{sect2new-def1}.
Such a result would be of great interest with respect to quantum studies considerations.
\end{remark}

We present first a positive result within the quite specific setting of continuous periodic functions on the real line.

\begin{theorem}\label{sect4-thm1}
Let $T>0$ and $\psi$ be a $T$-periodic continuous regular $\C$-supershift on $\R$. Then, there is a $T$-periodic entire function $F~: \C\rightarrow \C$ such that $\psi=F_{|\R}$.
\end{theorem}

\begin{proof}
Let
$$
\psi(a) = \sum\limits_{\kappa \in \Z} \gamma_\kappa \exp \Big( \frac{2i\pi \kappa}{T} \, a\Big)
$$
be the Fourier expansion of the $T$-periodic continuous function $\psi$ in $L^2(\R/T\Z)$. Here the spectrum $\widehat \psi = (\gamma_\kappa)_{k\in \Z}$ belongs to $\ell^2(\Z)$. Then the $T$-periodic continuous function
$$
\widetilde \psi : a \longmapsto \frac{1}{T}\int_0^T \psi(a-\tau)\, \psi(\tau)\, d\tau =
\sum\limits_{\kappa \in \Z} \gamma_\kappa^2\, \exp \Big( \frac{2i\pi\kappa}{T} \, a\Big)
$$
remains a $T$-periodic continuous $\C$-supershift on $\R$ (with spectrum
$(c_\kappa^2)_{\kappa \in \Z} \in \ell^1(\Z)$)
since one has for any
$\boldsymbol \epsilon \in [0,1[^{\N^*}$ such that $\lim_{N\rightarrow +\infty}
\epsilon_N=0$, for any
$(a,a') \in \R^2$ and any $N\in \N^*$, that
\begin{multline*}
\sum\limits_{\nu=0}^N \binom{N}{\nu}
\Big(\frac{1+a}{2}\Big)^{N-\nu}
\Big(\frac{1-a}{2}\Big)^\nu
\widetilde \psi \bigg( a' + \Big(1 - 2
\Big(\frac{\nu + \epsilon_N (N-\nu)}{N}\Big)\Big)\bigg) \\
=  \frac{1}{T}
\int_0^T
\Bigg(\sum\limits_{\nu=0}^N \binom{N}{\nu}
\Big(\frac{1+a}{2}\Big)^{N-\nu}
\Big(\frac{1-a}{2}\Big)^\nu
\psi \bigg( a'-\tau    +  \Big( 1 - 2
\Big(\frac{\nu + \epsilon_N (N-\nu)}{N}\Big)\Big)\bigg)\Bigg)\, d\tau.
\end{multline*}
Since $(a,a'-\tau)$ remains in the compact subset $K -\{0\}\times [0,T]$ when
$(a,a')\in K$ and $\tau\in [0,T]$, one needs just to use the fact that, according to the hypothesis on
$\psi$, for each family $\{\boldsymbol \epsilon_\iota\,:\,
\iota \in I'\}$ such that $\lim_{N\rightarrow +\infty} \epsilon_{\iota',N} =0$,
the sequence of functions
$$
\Bigg(
(\tau,a,a') \longmapsto
\sum\limits_{\nu=0}^N \binom{N}{\nu}
\Big(\frac{1+a}{2}\Big)^{N-\nu}
\Big(\frac{1-a}{2}\Big)^\nu
\psi \bigg( a'-\tau    +  \Big( 1 - 2
\Big(\frac{\nu + \epsilon_{\iota',N} (N-\nu)}{N}\Big)\Big)\bigg)
\Bigg)_{N\geq 1}
$$
converges uniformly on $[0,T] \times K$, where $K$ is any compact subset of $\R^2$, towards
the function $(\tau,a,a') \longmapsto \psi (a + a'-\tau)$, the convergence being uniform with respect
to the index $\iota'$. Since we know now that $\widetilde \psi$ is a continuous $\C$-regular supershift, it follows that for any $R\in \R_{>1}$, the sequence of functions
$$
\Big( a'\in [0,2\pi] \longmapsto
\sum\limits_{\nu=0}^N \binom{N}{\nu}
\Big(\frac{1+R}{2}\Big)^{N-\nu}
\Big(\frac{1-R}{2}\Big)^\nu \widetilde \psi \Big( a' + \Big(1-\frac{2\nu}{N}\Big)\Big)\Big)_{N\geq 1}
$$
converges uniformly towards $a'\longmapsto \widetilde \psi (a'+R)$ on $[0,2\pi]$.
One has for any $R>1$, any $a'\in [0,2\pi]$ and any $N\in \N^*$ that
\begin{multline}\label{sect4-eq12}
\sum\limits_{\nu=0}^N \binom{N}{\nu}
\Big(\frac{1+R}{2}\Big)^{N-\nu}
\Big(\frac{1-R}{2}\Big)^\nu \widetilde \psi \Big( a' + \Big(1-\frac{2\nu}{N}\Big)\Big) \\
=
\sum_{\kappa \in \Z} \gamma_\kappa^2
\bigg(
\sum\limits_{\nu=0}^N
\binom{N}{\nu} \Big(\frac{1+R}{2}\Big)^{N-\nu}
\Big(\frac{1-R}{2}\Big)^\nu
\exp \Big(i \kappa \Big(1- \frac{2\nu}{N}\Big)
\frac{2\pi}{T}\Big)\bigg)\, \exp \Big(\frac{2i\pi \kappa}{T}\, a'\Big) \\
= \sum_{\kappa \in \Z} \gamma_\kappa^2
\bigg( \cos \Big( \frac{1}{N} \, \frac{2\pi \kappa}{T}\Big) + iR\,
\sin \Big( \frac{1}{N} \, \frac{2\pi \kappa}{T}\Big)\bigg)^N\, \exp \Big(\frac{2i\pi \kappa}{T}\, a'\Big).
\end{multline}
Since the sequence of periodic functions of $a'$ defined by \eqref{sect4-eq12} converges uniformly
on $[0,2\pi]$ towards $a'\longmapsto \widetilde \psi (a'+R)$, it follows from Plancherel's formula that there exists a positive constant $C_{\widetilde \psi}(R)$ such that for any $M, N\in \N^*$,
\begin{equation}\label{sect4-eq12A}
\sum\limits_{\kappa = -N}^N
|\gamma_\kappa|^4 \bigg| \cos \Big( \frac{1}{MN} \, \frac{2\pi \kappa}{T}\Big) + iR\,
\sin \Big( \frac{1}{MN} \, \frac{2\pi \kappa}{T}\Big)\bigg|^{2MN}
\leq C_{\widetilde \psi}(R).
\end{equation}
Let $w= w_{\kappa,M,N} = 1/(MN) \times 2\pi\kappa/T$ for $-N\leq \kappa \leq N$. Observe that
\begin{equation}\label{sect4-eq13}
|w| \leq \frac{1}{M}\, \frac{2\pi}{T},\quad |(iR w - w^2/2)| \leq  |w| \Big( R + \frac{|w|}{2}\Big)
\leq \frac{1}{M} \frac{2\pi}{T} \Big( R + \frac{\pi}{T}\Big)
\end{equation}
for all such $w=w_{\kappa,M,N}$.
One has, as in \eqref{sect3-prop1bis-eq2}, that when $N\in \N^*$
tends to infinity
$$
\cos w + iR \sin w = 1-\frac{w^2}{2} + O(1/M^4) + i R\, w + R \, O(1/M^3)
$$
where the error terms $O(1/M^4)$ and $O(1/M^3)$ are uniform with respect to the choice of $\kappa$ in
$\{-N,...,N\}$ according to \eqref{sect4-eq13}.
Then, according once more to \eqref{sect4-eq13}, one has, as in \eqref{sect3-prop1bis-eq2},
\begin{equation}\label{sect4-eq14}
\begin{split}
\log \big(\cos w + iR\, \sin w\big) &= \log \bigg( 1 + iR\ w - \frac{w^2}{2}\bigg) + R\, O(1/M^3) \\
& = iR\, w -\frac{w^2}{2} - \frac{1}{2} \Big( iR\, w - \frac{w^2}{2}\Big)^2 + O((R/M)^3) \\
& = iR\, w + \frac{R^2-1}{2}\, w^2 + O((R/M)^3) \\
& = iR\, w + \frac{R^2-1}{2}\, w^2 \big(1 - O(R/M)\big).
\end{split}
\end{equation}
Let us choose $M = [\chi R]$, $\chi>0$ being a uniform constant large enough, such that the real part of the factor
$1-O(R/M)$ in \eqref{sect4-eq14} is bounded from below by $1/2$.
One has then
$$
\big|\cos w + iR\, \sin w\big|^{NM}
\geq \exp \Big( NM \frac{R^2-1}{4}\, w^2\Big).
$$
In particular, if $w = w_{\pm N,M,N} = \pm (2\pi/T)\times (1/M)$, one has
$$
\big|\cos w_{\pm N,M,N} + iR\, \sin w_{\pm N,M,N}\big|^{NM} \geq
\exp \Big( \frac{4 \pi^2}{T^2} \, \frac{R^2-1}{4 M}\, N\Big) =
\exp \Big( \frac{4 \pi^2}{T^2} \, \frac{R^2-1}{4 [\chi R]}\, N\Big)
$$
It then follows from \eqref{sect4-eq12A} that
$$
|\gamma_{\pm N}|\, \exp \Big( \frac{2 \pi^2}{T^2} \, \frac{R^2-1}{4 [\chi R]}\, N\Big) \leq (C_{\widetilde \psi}(R))^{1/4}.
$$
Since $R$ can be taken arbitrarily large and
$$
\lim\limits_{R \rightarrow +\infty} \frac{R^2-1}{[\chi\, R]} = +\infty,
$$
one concludes that for any $R'>0$,
$$
\sum\limits_{\kappa \in \Z} |\gamma_\kappa| \, e^{R' |\kappa|} < +\infty,
$$
from which it follows immediately that $\psi$ extend to $\C$ as an entire
$T$-periodic function.
\end{proof}

\section*{Declarations and statements}

{\bf Data availability}. The research in this paper does not imply use of data.

{\bf Conflict of interest}. The authors declare that there is no conflict of interest.


\begin{thebibliography}{99}
\bibitem{ABCS22} Y. Aharonov, J. Berndt, F. Colombo, P. Schlosser,
{\it A unified approach to Schr\"odinger evolution of superoscillations and supershifts},
Journal of Evolution Equations \textbf{22}, article number: 26 (2022).
\bibitem{ACSSST21}
Y. Aharonov, F. Colombo, I. Sabadini, T. Shushi, D.C. Struppa, J. Tollaksen, {\it A new method to generate superoscillating
functions and supershifts}, Proc. R. Soc. A. \textbf{477} (2249), Paper No. 20210020, 12 pp (2021).
\bibitem{ACJSSST22} Y. Aharonov,  F. Colombo, A.N. Jordan, I. Sabadini, T. Shushi, D.C. Struppa, J. Tollaksen, \textit{On superoscillations and supershifts in several variables},
Quantum Stud. : Math. Found., \textbf{9} (2022) : pp. 417-433.
\bibitem{ACSST13} Y. Aharonov, F. Colombo, I. Sabadini, D.C. Struppa, J. Tollaksen, \textit{On
the Cauchy problem for the Schr\"odinger equation with superoscillatory
initial data}, J. Math. Pure Appl., \textbf{99} (2013), (9),  pp. 165--173.
\bibitem{ACSST15} Y. Aharonov, F. Colombo, I. Sabadini, D.C. Struppa, J. Tollaksen,
\textit{Superoscillating sequences as solutions of generalized Schr\"odinger equations}, J. Math. Pure Appl. \textbf{103} (2015), pp. 522--534.
\bibitem{ACSST17A} Y. Aharonov,  F. Colombo,  I. Sabadini, D.C. Struppa, J. Tollaksen,
{\em The mathematics of superoscillations}, Memoirs of the American Mathematical Society
\textbf{257} (1174), American Mathematical Society, 2017.
\bibitem{ACSST17B} Y. Aharonov, F. Colombo, I. Sabadini, D. Struppa, J. Tollaksen,
\textit{Evolution of
superoscillatory initial data in several variables in uniform electric field}, J. Phys. A: Math.
Theor. \textbf{50} (2017).
\bibitem{ACSS18} T. Aoki, F. Colombo, I. Sabadini, D. C. Struppa, Continuity of some operators arising in the theory of superoscillations, Quantum Stud. Math. Found., \textbf{5} (2018), pp. 463-476.
\bibitem{Ber19} M. V. Berry \& al, {\it Roadmap on superoscillations}, Journal of Optics \textbf{21} (2019), (5), 053002.
\bibitem{bcss23} J. Behrndt, F. Colombo, P. Schlosser, D. C. Struppa, {\em Integral representation of superoscillations via complex Borel measures and their convergence},  Trans. Amer. Math. Soc. 376 (2023), 6315-6340.
\bibitem{Bern18} S. Bernstein,
\textit{Quelques remarques sur l'interpolation},
Math. Annalen \textbf{79} (1918), pp. 1-12.
\bibitem{Bern35} S. Bernstein,
\textit{Sur la convergence de certaines suites de polyn\^omes}, J. Math. Pures Appl. {\textbf 15}, no. 9 (1935), pp. 345-358.
\bibitem{BCSS14} R. Buniy, F. Colombo, I. Sabadini, D. C. Struppa, \textit{Quantum Harmonic Oscillator with
superoscillating initial datum}, J. Math. Phys. \textbf{55}, 113511 (2014).
\bibitem{CGS19} F. Colombo, J. Gantner, D. C. Struppa, \textit{Evolution by Schr\"odinger equation of Aharonov-Berry superoscillations in centrifugal potential}, Proc. Royal Soc. A. 475 (2019), no. 2225,
20180390, 17 pp.
\bibitem{CSSY21} F. Colombo, I. Sabadini, D.C. Struppa, A. Yger,
{\it Gauss sums, superoscillations and the Talbot carpet}, J. Math. Pures Appl. \textbf{147} (2021) pp. 163--178.
\bibitem{CSSY22} F. Colombo, I. Sabadini, D.C. Struppa, A. Yger,
{\it Superoscillating sequences and supershifts for families of generalized functions},
Complex Analysis and Operator Theory \textbf{16}, article number: 34 (2022).
\bibitem{cssy23} F. Colombo, I. Sabadini, D.C. Struppa, A. Yger, {\it Analyticity and supershift with regular sampling}, preprint, 2023.
\bibitem{Davis75} P. J. Davis, \textit{Interpolation and approximation},
Dover Publications, Inc., New York, 1975.
\bibitem{ErdV80} P. Erd\"os, P. V\'ertesi,
\textit{On the almost everywhere divergence of Lagrange interpolatory polynomials for arbitrary system of nodes}, Acta Mathematica Academiae Scientiarum Hungaricae
\textbf{36} (1-2), (1980), pp. 71-89.
Journal of Physics A: Mathematical and General, Volume \textbf{19} (1986), no. 10.
\bibitem{Kanto31} L. Kantorovitch,
\textit{Sur la convergence de la suite des polyn\^omes de S. Bernstein en dehors de l'intervalle
fondamental}, Bull. Acad. Sci. URSS (1931), pp. 1103-1115.
\bibitem{OrtW13} N. Ortner, P. Wagner,
\textit{Distribution-Valued Analytic Functions-Theory and Applications},
edition Swk, Hamburg, 2013.
\bibitem{Phil03}
G. M. Phillips, \textit{Interpolation and Approximation by Polynomials}, CMS Books in Mathematics, Springer Science + Business Media New York, 2003.
\bibitem{DGP63}
F.G. Dressel, J.J. Gergen, W.H. Purcell Jr.,
\textit{Convergence of extended Bernstein polynomials in the complex plane},
Pacific J. Math. \textbf{13}, no. 4 (1963), pp. 1171-1180.
\bibitem{Lor53} G.G. Lorentz,
\textit{Bernstein polynomials}, Toronto, 1953 (second edition : Chelsea Publishing Company, New York, 1986).
\bibitem{Neum1862}
K. Neumann, \textit{Ub\"er die Entwickelung Einer Funktion Nach den Kugelfunktionen}, Journal f\"ur Mathematik, 1862.
\bibitem{Rev2000} M. Revers,
\textit{The divergence of Lagrange interpolation for $|x|^\alpha$ at equidistant nodes},
Journal of Approximation Theory \textbf{100} (2000), no. 2, pp. 269-280.
\bibitem{Tayl68}
B.A. Taylor, {\it Some locally convex spaces of entire functions}, p. 431-467 in
{\em Entire functions and related parts of analysis},
Proceedings of Symposia in Pure Mathematics \textbf{11}, American Mathematical Society, Providence, 1968.
\bibitem{Schiff93}
J. L. Schiff, \textit{Normal families}, Universitext, Springer-Verlag, New York, 1993.
\bibitem{Titch39} E. C. Titchmarsh, {\it The theory of functions}, Oxford, 1939.
\end{thebibliography}
\end{document}